\documentclass[12pt,a4paper]{amsart}
\usepackage{times}
\usepackage{amscd,amssymb,amsthm,latexsym,amsmath}
\usepackage{amsfonts,latexsym,amssymb,url}
\usepackage{hyperref}
\usepackage{pifont}

\newtheorem{teo}{Theorem}[section]
\newtheorem{cor}{Corollary}[section]
\newtheorem{lemma}{Lemma}[section]
\newtheorem{prop}{Proposition}[section]
\newtheorem{defi}{Definition}[section]
\newtheorem{remark}{Remark}[section]

\newcommand{\R}{{\mathbb R}}

\newcommand{\U}{\mathcal{U}}
\newcommand{\med}{{\int\!\!\!\!\!\!\!\!-\!\! -\!\!\!\!}}
\DeclareMathOperator{\loc}{loc}

\DeclareMathOperator*{\osc}{osc} 
\DeclareMathOperator{\diam}{diam}

\begin{document}
   \title[Harnack inequality and ....]
   {Harnack inequality and regularity for degenerate quasilinear elliptic equations}
   \author{G. Di Fazio}
   \address{Dipartimento di Matematica e Informatica\\ Universit\`a di Catania\\
   Via\-le A. Doria 6, 95125, Catania, Italy}
   \email{difazio@dmi.unict.it}
\author{M. S. Fanciullo}
   \address{Dipartimento di Matematica e Informatica\\ Universit\`a 
   di Catania\\
   Via\-le A. Doria 6, 95125, Catania, Italy}
   \email{fanciullo@dmi.unict.it}

    \author{P. Zamboni}
   \address{Dipartimento di Matematica e Informatica\\ Universit\`a 
   di Catania\\
   Via\-le A. Doria 6, 95125, Catania, Italy}
   \email{zamboni@dmi.unict.it}
   
 \keywords{Harnack inequality, Strong $A_\infty$ weights, degenerate elliptic equations, Stummel class, Morrey spaces}
   \subjclass[2000]{}
   \thanks{}
\date{\today}

\linespread{1.4}

\begin{abstract}
We prove Harnack inequality and local regularity results for weak solutions of a quasilinear degenerate equation in divergence form under natural growth conditions. 
The degeneracy is given by a suitable power of a strong $A_\infty$ weight.
Regularity results are achieved under minimal assumptions on the coefficients and,
as an application,
we prove $C^{1,\alpha}$ local estimates for solutions of a degenerate equation in non divergence form.
\end{abstract}

\maketitle

\tableofcontents

\section{Introduction}

This paper contains a contribution towards a complete theory concerning regularity and smoothness for solutions of degenerate elliptic equations under minimal assumptions on the coefficients. Here we consider quasilinear elliptic equations whose ellipticity degenerates as a suitable power of a strong $A_\infty$ weight.
The class of strong $A_\infty$ weights has been introduced by David and Semmes in \cite{DS} and it is useful in several problems related to geometric measure theory and quasiconformal mappings.
Indeed, as it is well known, the Jacobian of a quasiconformal mapping is a strong $A_\infty$ weight.
Weights of this kind enjoy some metric properties and important inequalities like Poincar\`e and Sobolev's.
Moreover any strong $A_\infty$ weight is a Muckenhoupt weight, and there exist $A_2$ weights which are not strong $A_\infty$ weights. 
The only Muckenhoupt weights whose degeneration gives regularity for elliptic equations are those in the class $A_2$
(see e.g. \cite{FKS}, \cite{GUT}, \cite{vz}, \cite{z}, \cite{dfz}).

Let us consider quasilinear elliptic equations in divergence form
\begin{equation} \label{equazione:introduzione} 
{\rm div} A(x,u,\nabla u)+B(x,u,\nabla u)=0\,,
\end{equation}
where $A$ and $B$ are measurable functions satisfying suitable structure conditions 
\begin{equation} \label{ipo:strutt:introduzione}
\begin{cases}
|A(x,u,\xi)|\le a\omega(x)|\xi|^{p-1}+b(x)|u|^{p-1}+e(x) & \\
|B(x,u,\xi)|\le b_0\omega(x)|\xi|^{p}+b_1(x)|\xi|^{p-1}+d(x)|u|^{p-1}+f(x) & \\
\xi\cdot A(x,u,\xi)\ge \omega |\xi|^p-d(x)|u|^p-g(x)\,. & \\
\end{cases}
\end{equation}
Here $1<p<n$, $\omega=v^{1-\frac{p}{n}}$, $v$ is a strong $A_\infty$ weight and the coefficients of the lower order terms belong to suitable Stummel - Kato or Morrey classes.
The function $B$ is required to have natural growth in the variable $\xi$.

Equation \eqref{equazione:introduzione} with $v\equiv 1$ has extensively been investigated. 
Here we quote some contributions - among others - by Trudinger and Lieberman.
In \cite{t} (see also \cite{gt}) Trudinger considers the same equation with no degeneracy and coefficients in suitable $L^p$ classes. 
There, Harnack inequality and regularity properties of bounded weak solutions are proved. 
In \cite{Lieberman} Lieberman considers equation 
\begin{equation} \label{equazione:lieberman:introduzione} 
{\rm div} A(x,u,\nabla u)+B(x,u,\nabla u)=\mu
\end{equation}
assuming $\mu$ to be a given signed Radon measure satisfying a Morrey type condition.
There, Harnack inequality and regularity for bounded weak solutions are proved under the structure conditions \eqref{ipo:strutt:introduzione} with lower order terms in suitable Morrey classes.

Our results are parallel to those in \cite{t} and \cite{Lieberman}. We follow the pattern drawn in \cite{t}.

The novelty here is the special kind of degeneracy.
We assume the coefficients in suitable Stummel-Kato and Morrey classes and - as a technique - use a Fefferman type inequality proved in \cite{dz} to control the integrals arising from the lower order terms. The inequality is based on a representation formula proved in \cite{FGW} (see also \cite{FRANCHI-LINCEI}).

In Section 4, as an application of the previous results, we prove $C^{1,\alpha}$ estimates for a non variational elliptic equation related to equation \eqref{equazione:introduzione}.

\section{Strong $A_\infty$ weights and function spaces}

Let $v$ be an $A_\infty$ weight in $\R^n$.
This means that, for any $\varepsilon>0$ there exists
$\delta>0$, such that if $Q$ is a cube in $\R^n$ and $E$ is a measurable subset of $Q$ for which
$|E|\leq \delta |Q|$ holds, then $v(E)\leq \varepsilon  v(Q)$, i.e. $\int_E v(x)\,dx \leq \varepsilon \int_Q v(x)\,dx$.
If $v \in A_\infty$ and $B_{x,y}$ is the euclidean ball containing $x$ and $y$ with diameter $|x-y\,|$, we can define a quasi distance $\delta$ in $\R^n$ by setting
\begin{equation*}
\delta(x,y) = \left(\int_{B_{x,y}} v(t)\,dt\right)^{1/n}\!\!\!\!\!\!.
\end{equation*}
We remark that $\delta(x,y)=|x-y\,|$ when $v(t)\equiv 1$. By using the function $\delta(x,y)$ we may define the $\delta$-length of a curve as the limsup of the $\delta$-lengths of the approximating polygonals.

On the other side we can actually define a distance related to the weight $v$.
We take, as the distance between two points $x$ and $y$,
the infimum of the $\delta$-length of the curves connecting $x$ and $y$.
Namely we set,
\begin{equation*}
d_v(x,y)
=
\inf \{\hbox{$\delta$-length of the curves connecting $x$ and $y$}\}\,.
\end{equation*}
In general, the function $\delta$ is not comparable to a distance.

\begin{defi}
If $v$ is an $A_\infty$ weight there exists a positive constant $c$ such that
$\delta(x,y) \leq c\, d_v(x,y)$,
for any $x$, $y \in \R^n$ \hbox{\rm(}see \cite{DS}\hbox{\rm)}.
If, in addition,
\begin{equation} \label{eq:comparability}
\delta(x,y) \sim d_v(x,y) \quad \forall x,y \in \R^n
\end{equation}
we say that $v$ is a strong $A_\infty$ weight.
\end{defi}

The measure $v\,dx$ is Ahlfors regular and, as a consequence, is a doubling measure (see e.g. \cite{SEMMES}).

In this section we denote by $B\equiv B(x,R)$ and $B_e\equiv B(x,R)$ respectively the metric and euclidean balls centered at $x$ with radius $R$ .

\begin{teo}
Let $v$ be a strong $A_\infty$ weight.
Then, there exist two positive constants $a$ and $A$, depending only on $n$ and the comparability constants in \eqref{eq:comparability},
such that for any $x\in \R^n$ and any $r>0$, we have
\begin{equation*}
a\,r^n \leq v(B(x,r)) \leq A\, r^n\,.
\end{equation*}
Moreover, there exists $c>0$ such that for any $r>0$ there exists $R=R(r)$ such that
\begin{equation*}
B_e(x,cR) \subseteq B(x,r) \subseteq B_e(x,R) \qquad \forall x\in \R^n\,.
\end{equation*}
\end{teo}
It is possible to compare the $A_p$ classes of Muckenhoupt weights and strong $A_\infty$.
\begin{remark}
Any strong $A_\infty$ weight is a $A_\infty$ weight. For any $1<p<\infty$ there exists
an $A_p$ weight which is not a strong $A_\infty$ weight.
\end{remark}

In \cite{DS} David and Semmes show that Poincar\'e and Sobolev inequalities hold true for
strong $A_\infty$ weights. Unfortunately, they prove inequalities with different weights on both sides and, in order to run Moser iteration, we need Poincar\'e and Sobolev inequalities with the same weight on both sides.
However, in \cite{HEI-KOSK} Sobolev and Poincar\'e
inequalities are proved as a consequence of the results in \cite{DS} and \cite{FGW}, and
in \cite{FRANCHI-HAILASZ} it is shown how to pass from a Poincar\'e inequality with two different weights
to a Poincar\'e inequality with the same weight on both sides.

We quote only the results in the form we need. Our statement can be easily
derived from the above cited papers.

\begin{teo}[\cite{HEI-KOSK}]
Let $v$ be a strong $A_\infty$ weight and $1<p<n$. Let $q$ be such that $v \in A_q$.
The following Sobolev inequality and Poincar\'e inequality hold true
\begin{equation} \label{heinonen-koskela}
\left(\med_B |u(x)|^{kp}\,\omega\,dx\right)^{\frac{1}{kp}}
\leq
c \,\diam(B) \left(\med_B |\nabla u(x)|^{p}\,\omega\,dx\right)^{\frac{1}{p}}\!\!\!
\
\forall u \in C^\infty_0(B)
\end{equation}

\begin{equation} \label{heinonen-koskela-poincare}
\med_B |u - u_B|^p\,\omega\,dx \leq c (\diam B)^p \med_B |\nabla u|^p\,\omega\,dx
\qquad
\forall u \in C^\infty(B)
\end{equation}
where $\omega(x\,)=v(x\,)^{1-\frac{p}{n}}$, and \hbox{$\ \med \ $ } denotes the average
with respect to the measure $\omega(x\,)\,dx$,
$k= \dfrac{p+q(n-p)}{q(n-p)}$ and
$B$ denotes a metric ball.
\end{teo}

Using strong $A_\infty$ weights we define Lebesgue and Sobolev classes.
\begin{defi}
Let $v$ be a strong $A_\infty$ weight and $\omega = v^{1-p/n}$, $\Omega \subset \R^n$.
For any $u\in C^\infty_0(\Omega)$ we set
\begin{equation} \label{Lp-pesato}
\|u\|_{p,v} = \left(\int_\Omega |u(x)|^p\,\omega(x)\,dx\right)^{1/p}\!\!\! \quad 1\leq p<\infty\,.
\end{equation}
We define $L^p_v(\Omega)$ to be the completion of $C^\infty_0(\Omega)$ with respect to the above norm.
In a similar way we define Sobolev classes.
Let $1<p<n$.
For any $u\in C^\infty(\Omega)$ we set
\begin{equation} \label{spazio-di-sobolev-pesato}
\|u\|_{1,p,v} = \left(\int_\Omega |u(x)|^p\,\omega(x)\,dx\right)^{1/p}
+
\left(\int_\Omega |\nabla u(x)|^p\,\omega(x)\,dx\right)^{1/p}\!\!\!\,.
\end{equation}
We define $H^{1,p}_{0,v}(\Omega)$ to be the completion of
$C^\infty_0(\Omega)$ with respect to the above norm and
$H^{1,p}_{v}(\Omega)$ to be the completion of
$C^\infty(\Omega)$ with respect to the same norm.
\end{defi}

\begin{remark}
In the above definitions we put $v$ in the symbol of the norm and $\omega$ into the
integrals. This is because we want to stress the dependence on the strong $A_\infty$ weight $v$.
\end{remark}

\begin{remark}
In general, i.e. if $\omega\not\in A_2$, the classes $H$ and $W$ are different. Here we are going to study regularity of weak solutions taking $H$ as a class of test functions (see \cite{FSSC} and \cite{SC}).
\end{remark}

In order to formulate the assumptions on the lower order terms we need to define some other function spaces.
\begin{defi} \label{def:stummel}
Let  $f$ be a locally integrable function in $\Omega \subset \R^n$ and let $v$ be a strong $A_\infty$
weight.
We set
\begin{equation} \label{phi-function}
\phi(f;R)=\sup_{x\in \Omega}\!
\left(
\int_{B(x,R\,)}
\!\!\!\!\!\!
k(x,y)
\left(\int_{B(x,R\,)} 
\!\!\!\!\!\!\!\!\!\!\!\! 
|f(z)| k(z,y) \,v(z)^{1-\frac{p}{n}}
\,dz\right)^{\frac{1}{p-1}} \!\!\! \!\!\!  v(y)\,dy
\right)^{p-1}
\end{equation}
where 
$$
k(x,y) = \dfrac{1}{v(B(x,d_v(x,y)))^{1-\frac{1}{n}}}\,.
$$
We shall say that $f$ belongs to the class $\tilde{S}_v(\Omega)$ if
$\phi(f;R)$ is bounded in a neighborhood of the origin.
Moreover, if
$\displaystyle{\lim_{R\to 0} \phi(f;R) = 0}$ 
then we say that $f$ belongs to the Stummel-Kato class
$S_v(\Omega)$.
If there exists $\rho > 0$ such that
\begin{equation} \label{ipotesi_Dini}
\int_0^{\rho}\displaystyle\frac{\phi(f;t)^{1/p}}{t}\,dt
< +\infty\,,
\end{equation}
then we say that the function $f$ belongs to the class $S_v^{'}(\Omega)$.
\end{defi}
\begin{defi} [Morrey spaces]
Let $p \in [1,+\infty[$ and $v$ be a strong $A_\infty$ weight. We say that
$f$ belongs to $L_v^{p,\lambda}(\Omega)$, for some $\lambda >0$, if
$$
\|f\|_{L_v^{p,\lambda}(\Omega)}
= $$
$$=\sup\limits_{x \in \Omega,0<r<d_0}
\left( {\frac{r^{\lambda}}{  |B(x,r) \cap \Omega|}}
\int\limits_{B(x,r) \cap \Omega}
|f(y)|^p v(y)^{1-\frac{p}{n}}\, dy
\right)^{\frac{1}{p}}
 < + \infty,
$$
where $d_0 = \hbox{\rm diam}(\Omega)$.
\end{defi}
\begin{remark}
It is an easy task to check that the above definitions give back their classical counterparts
when $v \equiv 1$.
\end{remark}

It is easy to compare the function classes previously defined.

\begin{prop} \label{embedding_Stummel_Morrey}
Let $p$ and $\varepsilon$ be numbers such that $1<p<n$ and $0< \varepsilon <p$.
The class $L_v^{1,p-\varepsilon}(\Omega)$ is embedded in $S_v^{'}(\Omega)$.
Namely, if $V \in L_v^{1,p-\varepsilon}(\Omega)$ then
$$
\phi\left(V;r\right) \le C \|V\|_{L_v^{1,p-\varepsilon}} r^{\frac{\varepsilon}{p-1}}
$$
for any $0 < r < \diam{\Omega}$.
\end{prop}

\begin{proof}
The proof is standard and can be easily adapted from the case $\omega=1$ (see \cite{DIFAZIO-PADOVA}).
\end{proof}
The following two lemmas will be useful in the iteration process.

\begin{lemma}  [\cite{RAG-ZAM}] \label{lemmaFilippo}
Let $\mu(r)$ a continuous positive increasing function defined in
$]0,+\infty[$ such that $\lim\limits_{r \to 0} \mu(r) = 0$,
$0<\theta<1$.
The series
\begin{equation*}
\sum\limits_{j=0}^{+\infty}\theta^j \log\,\mu^{-1}
\left( \theta^{2j}\right)
\end{equation*}
is convergent if and only if there exists $\rho > 0$ such that condition \eqref{ipotesi_Dini} is satisfied.
\end{lemma}

\begin{lemma} [\cite{RAG-ZAM}]  \label{lemma-stampacchia}
Let $0<\gamma<1$, $h:\left]0,+\infty\right[\to
\left]0,+\infty\right[$ a non decreasing function with
$\lim\limits_{t\to0}h(t)=0$,
such that
$$
h(t)\le C h(t/2)\qquad\qquad (C>1)
$$
and $\omega:\left]0,+\infty\right[\to\left]0,+\infty\right[$ a non decreasing
function.

If
$$
\omega(\rho)\le\gamma\omega(4\rho)+h(\rho)
\qquad\forall\rho<\rho_0<1
$$
then there exist $\overline\rho\le\rho_0$, $0<\sigma\le1$
and a positive constant
$K$ such that
$$
\omega(\rho)\le Kh^{\sigma}(\rho)\qquad\forall\rho<\overline\rho.
$$
\end{lemma}

The following result will be useful in the proof of the weak Harnack inequality.

\begin{teo} [\cite{dz}]  \label{embedding-con-peso}
Let $\Omega$ be a bounded domain in $\R^n$ and let $V$ belong to the class $\tilde{S}_v(\Omega)$.
If $v$ is a strong $A_{\infty}$ weight and $1<p<n$, then there exists a constant $c$
such that for any $u \in C^\infty_0(\Omega)$ we have
\begin{equation} \label{embedding}
\left(\int_B |V(x)| |u(x)|^p \, \omega\,dx\right)^{1/p}
\leq
c \,\phi^{1/p}\left(V;2R\right)
\left(\int_B |\nabla u(x)|^p \omega \,dx\right)^{1/p}
\end{equation}
where $\omega(x\,)\equiv v^{1-\frac{p}{n}}(x\,)$ and $R$ is the radius of a metric ball
$B \equiv B_R$, containing the support of $u$.
\end{teo}

As a direct consequence we have
\begin{cor} \label{cor:pezzopiccolo-pezzogrande}
Let $1<p<n$ and $v$ be a strong $A_\infty$ weight.
Let $V$ belongs to the class $S_v(\Omega)$.
Then, for any $\varepsilon >0$, there exists $K(\varepsilon)$ such that
\begin{multline} \label{pezzopiccolo-pezzogrande}
\int_\Omega |V(x)| |u(x)|^p \omega(x)\,dx
\leq
\varepsilon \int_\Omega |\nabla u(x)|^p \omega(x)\,dx 
+
\\
+
K(\varepsilon) \int_\Omega |u(x)|^p \omega(x)\,dx
\ 
\forall u \in C^\infty_0(\Omega)
\end{multline}
where $\omega(x\,) = v(x\,)^{1-\frac{p}{n}}$,
$K(\varepsilon) \sim \dfrac{\sigma}{\left[\phi^{-1}\left(V;\varepsilon\right)\right]^{n+p}}$
and $\phi^{-1}$ denotes the inverse function of $\phi$.
\end{cor}

\section{Harnack inequality}

In this section we shall prove a weak Harnack inequality for non negative weak solutions of the equation 
\begin{equation}\label{equazione}
{\rm div} A(x,u,\nabla u)+B(x,u,\nabla u)=0\,.
\end{equation}
We recall what we mean by weak solution of \eqref{equazione}.
\begin{defi}
A function $u\in H^{1,p}_v(\Omega)$ is a local weak subsolution (supersolution) of equation \eqref{equazione} in $\Omega$
if
\begin{equation} \label{def:soluzione}
\int_\Omega A(x,u(x),\nabla u(x))\cdot \nabla \varphi \,dx
-
\int_\Omega B(x,u(x),\nabla u(x)) \varphi \,dx 
\leq 0 \ (\geq 0 )
\end{equation}
for every $\varphi \in H^{1,p}_{0,v}(\Omega)$.
A function $u$ is a weak solution if it is both super and sub solution.
\end{defi}

We require the functions $A(x,u,p)$ and $B(x,u,p)$ to be measurable functions satisfying the following structure conditions
\begin{equation} \label{ipo:strutt}
\begin{cases}
|A(x,u,\xi)|\le a\omega(x)|\xi|^{p-1}+b(x)|u|^{p-1}+e(x) & \\
|B(x,u,\xi)|\le b_0\omega(x)|\xi|^{p}+b_1(x)|\xi|^{p-1}+d(x)|u|^{p-1}+f(x) & \\
\xi\cdot A(x,u,\xi)\ge \omega |\xi|^p-d(x)|u|^p-g(x) & \\
\end{cases}
\end{equation}
where $1<p<n$, $\omega=v^{1-\frac{p}{n}}$ and $v$ is a strong $A_\infty$ weight.

We shall show that locally bounded weak solutions verify a Harnack inequality and, as a consequence, some regularity properties. We shall make the following assumptions on the lower order terms to ensure the continuity of local weak solutions
\begin{equation} \label{ipo:terminiinferiori}
a, b_0\in \R, \left(\frac{b}{\omega}\right)^{\frac{p}{p-1}}, \left(\frac{b_1}{\omega}\right)^p, \frac{d}{\omega}, \left(\frac{e}{\omega}\right)^{\frac{p}{p-1}}, \frac{f}{\omega}, \frac{g}{\omega}\in S^\prime_v(\Omega)\,.
\end{equation}

>From now on we denote by $B_r=B_r(x)$ the euclidean ball centered at $x$ with radius $r$. 
\begin{teo}\label{harnackinterno}
Let $u$ be a non negative weak supersolution of equation \eqref{equazione} in $\Omega$ satisfying \eqref{ipo:strutt} 
and \eqref{ipo:terminiinferiori}. 
Let $B_r$ be a ball such that $B_{3r} \Subset \Omega$ and let $M$ be a constant such that $ u\le M$ in $B_{3r}$. 
Then there exists $c$ depending on $n$, $M$, $a_0$, $b_0$, $p$ and the weight $v$ such that  
$$
\omega^{-1} (B_{2r})\int_{B_{2r}}u\,\omega dx\le c \left\{ {\rm min}_{B_r}u + h(r)\right\}
$$
where 
$\displaystyle{
h(r)={\left[
\phi\left(\left(\frac{e}{\omega}\right)^{\frac{p}{ p-1}};3r\right)
+
\phi\left(\frac{g}{\omega};3r\right) \right]}^{\frac{1}{
p}}
+
\left[ \phi\left(\frac{f}{\omega};3r\right) \right]^{\frac{1}{p-1}}\,.
}$
\end{teo}

\begin{proof} 
We simplify the structure assumptions by setting $ w = u+h(r)$.
 We get
\begin{equation}
\begin{cases}\label{ipotesi_di_struttura_ridotte}
\hskip 5pt |A(x,u,\xi)|
\le
&
\!\!\!
a\omega(x)|\xi|^{p-1}+b_2(x)|w|^{p-1}
\\
\hskip 5pt
|B(x,u,\xi)|
\le
&
\!\!\!
b_0\omega(x)|\xi|^p+b_1(x)|\xi|^{p-1}+d_1(x)|w|^{p-1}
\\
\xi \! \cdot \!\! A(x,u,\xi)
\ge
&
\!\!\!\!\!
\omega(x)|\xi|^p-d_1(x)|w|^p
\end{cases}
\end{equation}
where $b_2=b+h^{1-p}e$ and $d_1=d+h^{1-p}f+h^{-p}g$.
Is is easy to check that $b_2$ and $d_1$ verify the same assumptions of $b$ and $d$.

We take $\varphi=\eta^p w^\beta e^{- b_0w}$, $\beta<0$ as test function in \eqref{def:soluzione} so we obtain
\begin{multline*}
\int_{B_{3r}}\eta^p e^{-b_0w} (b_0w^\beta+|\beta|w^{\beta-1})\nabla w \cdot A-
\\
p\int_{B_{3r}}w^\beta\eta^{p-1}e^{-b_0w}\nabla \eta\cdot A+\int_{B_{3r}}\eta^pw^\beta e^{- b_0w} B\le 0\,.
\end{multline*}
The previous inequality and the structure assumptions \eqref{ipotesi_di_struttura_ridotte} yield

\begin{multline*}
\int_{B_{3r}}e^{- b_0 w}\eta^p(b_0w^\beta+|\beta| w^{\beta-1})|\nabla w|^p\omega dx \le
\\
\int_{B_{3r}}e^{- b_0w}\eta^p(b_0w^\beta+|\beta| w^{\beta-1})(\nabla w\cdot A+ d_1 |w|^p)dx\le
\\
p\int_{B_{3r}}w^{\beta}\eta^{p-1}e^{-  b_0w}\nabla \eta\cdot A\,dx
-
\int_{B_{3r}}\eta^pw^{\beta}e^{-  b_0w}B\,dx+
\\
\int_{B_{3r}}e^{- b_0w}\eta^p(b_0w^\beta +|\beta| w^{\beta-1})d_1|w|^p\,dx\le
\\
p\int_{B_{3r}}w^{\beta}\eta^{p-1}e^{-b_0w}\nabla \eta\cdot A\,dx+
\\
\int_{B_{3r}}\eta^pw^{\beta}e^{-b_0w}(b_0|\nabla w|^p\omega+b_1|\nabla w|^{p-1}+d_1|w|^{p-1})dx+
\\
\int_{B_{3r}}e^{- \beta b_0w}\eta^p(b_0w^\beta+|\beta| w^{\beta-1})d_1|w|^p\,dx\,.
\end{multline*}

By Young inequality and boundedness of $w$ in $B_{3r}$ we obtain

\begin{multline*}
|\beta|\int_{B_{3r}}\eta^pw^{\beta-1}|\nabla w|^p\omega dx 
\le
c p\int_{B_{3r}}w^{\beta}\eta^{p-1}\nabla \eta\cdot A\,dx+
\\
c\int_{B_{3r}}\eta^pw^{\beta}(b_1|\nabla w|^{p-1}+d_1|w|^{p-1})dx+
\\
c\int_{B_{3r}}\eta^p(b_0w^\beta +|\beta| w^{\beta-1})d_1|w|^p\,dx\le
\\
c \int_{B_{3r}}\Big\{pw^{\beta}\eta^{p-1}|\nabla \eta|(a\omega|\nabla w|^{p-1}+b_2|w|^{p-1})+
\\
\eta^pw^\beta b_1|\nabla w|^{p-1}+\eta^pw^{\beta+p-1}d_1+
\\
+
\eta^pb_0w^{\beta+p}d_1+|\beta|\eta^pw^{\beta+p-1}d_1\Big\}dx\le
\\
c \int_{B_{3r}}\Big\{pw^{\beta}\eta^{p-1}|\nabla \eta|a \omega|\nabla w|^{p-1}
+
pw^{\beta+p-1}\eta^{p-1}|\nabla \eta|b_2+
\\
\eta^pw^\beta b_1|\nabla w|^{p-1}
+
(1+|\beta|)\eta^pw^{\beta+p-1}d_1+ \eta^pb_0w^{\beta+p}d_1\Big\}dx\,.
\end{multline*}

Then,
\begin{multline*}
|\beta|\int_{B_{3r}}\eta^pw^{\beta-1}|\nabla w|^p\omega dx 
\le 
\\
\le
c(b_0,M,p) \int_{B_{3r}}\Big\{w^{\beta}\eta^{p-1}|\nabla \eta| a |\nabla w|^{p-1}\omega\, dx 
+ 
\\
+
\epsilon\eta^pw^{\beta-1}|\nabla w|^p\omega+c(\epsilon)\eta^p\frac{b_1^p}{\omega^{p-1}}w^{\beta+p-1}+
\\
+
\eta^{p-1}|\nabla \eta|w^{\beta+p-1}b_2+(1+|\beta|)\eta^pw^{\beta+p-1}d_1\Big\}dx
\le
\\
\le
c(b_0,M,a,p)\int_{B_{3r}}\Big\{w^{\beta}\eta^{p-1}|\nabla \eta| |\nabla w|^{p-1}\omega dx 
+
\\
+\epsilon\eta^pw^{\beta-1}|\nabla w|^p\omega+c(\epsilon)\eta^p\frac{b_1^p}{\omega^{p-1}}w^{\beta+p-1}
+
\\
+w^{\beta+p-1}|\nabla \eta|^p\omega+\eta^pw^{\beta+p-1}\frac{b_2^{\frac{p}{p-1}}}{\omega^{\frac{1}{p-1}}}
+
\\
+(1+|\beta|)\eta^pw^{\beta+p-1}d_1\Big\}dx\,.
\end{multline*}
We set $V=\frac{b_2^{\frac{p}{p-1}}}{\omega^{\frac{1}{p-1}}}+ \frac{b_1^p}{\omega^{p-1}}+d_1$
in order to get short the previous inequality. 
We obtain
\begin{multline}\label{eq:dopotest}
\int_{B_{3r}}\eta^p w^{\beta-1}|\nabla w|^p\omega dx 
\le
\\
\le
c( 1+|\beta|^{-1} )^p 
\int_{B_{3r}}\Big\{|\nabla \eta|^p w^{\beta+p-1}\omega + V\eta^p w^{\beta+p-1}\Big\}dx\,.
\end{multline}

Now the proof follows the lines of Theorem 4.3 in \cite{dz}.
We set

\begin{equation*}
\U(x) =
\begin{cases}
w^q(x)\quad
&
\quad \hbox{where} \quad pq=p+\beta-1 \quad \hbox{if} \quad \beta\neq 1-p
\\
\log w(x)
&
\quad \hbox{if} \quad \beta=1-p
\end{cases}
\end{equation*}

by \eqref{eq:dopotest} we have
\begin{multline} \label{eq:betadiverso}
\int_{B_{3r}}\eta^p|\nabla \U|^p\omega(x)\,dx
\le
c|q|^p(1+|\beta|^{-1})^p
\left\{\int_{B_{3r}}|\nabla \eta|^p\U^p \omega(x)\,dx + \right.
\\
\left.
+
\int_{B_{3r}}V\eta^p \U^p\,dx\right\}\,,
\beta \neq 1-p
\end{multline}
while
\begin{equation} \label{eq:betauguale}
\int_{B_{3r}}\eta^p|\nabla \U|^p\omega(x)\,dx
\le
c\left\{\int_{B_{3r}}|\nabla\eta|^p\omega(x)\,dx
+
\int_{B_{3r}}V\eta^p\,dx\right\}
\end{equation}
if  $\beta=1-p$.
Let us start with the case $\beta = 1-p$. By Theorem \ref{embedding-con-peso} we have
\begin{equation*}
\int_{B_{3r}}V\eta^p\,dx
\le
c\phi\left(\frac{V}{\omega}; \diam{\Omega}\right) \int_{B_{3r}}|\nabla\eta|^p\omega(x)\,dx\,,
\end{equation*}
and from \eqref{eq:betauguale}
\begin{equation*}
\int_{B_{3r}}\eta^p|\nabla \U|^p\omega(x)\,dx
\le
c \int_{B_{3r}}|\nabla \eta|^p\omega(x)\,dx\,.
\end{equation*}

Let $B_h$ be a ball contained in $B_{2r}$. Choosing
$\eta(x)$ so that $\eta(x)=1$ in $B_h$, $0\le\eta\le1$ in
$B_{3r}\setminus B_h$ and $|\nabla \eta|\le \displaystyle\frac{3}{h}$,
we get
\begin{equation*}
\|\nabla \U\|_{L^p_v(B_h)}
\le
c
\dfrac{\omega(B_h)^{\frac{1}{ p}}}{ h}\,.
\end{equation*}
By Poincar\'e inequality \eqref{heinonen-koskela-poincare} and John--Nirenberg lemma
(see \cite{BUCKLEY}) we have $\U(x) = \log w(x) \in BMO_v$.
Then there exist two positive constants $p_0$ and $c$, such that
\begin{equation} \label{eq:pesoexp}
\left(\med_{B_{2r}}e^{p_0 \U}\omega(x)\,dx\right)
^{\frac{1}{ p_0}} \left(\med_{B_{2r}} e^{-{p_0 \U}} \omega(x)\,dx\right) ^{\frac{1}{ p_0}} \le c\,.
\end{equation}

Let us consider the following family of seminorms
\begin{equation*}
\Phi(p,h)=\left(\int_{B_h}|w|^p\omega(x)\,dx\right)^{1/p}\,,
\quad p \neq 0\,.
\end{equation*}
By \eqref{eq:pesoexp} we have
\begin{equation*}
\dfrac{1}{\omega(B_{2r})^{1/p_0}}\Phi(p_0,2r)
\le
c\omega(B_{2r})^{1/p_0} \Phi(-p_0,2r)\,.
\end{equation*}
In the case \eqref{eq:betadiverso} by Corollary \ref{cor:pezzopiccolo-pezzogrande} we obtain
\begin{multline}
\int_{B_{3r}}|\nabla\U|^p\eta^p\omega(x)\,dx
\le
c \left\{
\!
 (|q|^p +1) \!\! \left(1+{\frac{1}{ |\beta|}}\right)^p 
\!\!\!\!
\int_{B_{3r}}|\nabla \eta|^p\U^p\omega(x)\,dx + \right.
\\
\left.
+
\left[{\frac{1}{ \phi^{-1}\left(\frac{V}{\omega};|q|^{-p}\left(1+{\frac{1}{
|\beta|}}\right)^{-p}\right) }}\right]^{n+p}
\int_{B_{3r}}\eta^p\U^p\omega(x)\,dx\right\}\,.
\end{multline}
By Sobolev inequality we have
\begin{multline}
\left(\int_{B_{3r}}|\eta \U|^{kp}\omega(x)\,dx\right)^{\frac{1}{k}}
\le
c \omega(B)^{\frac{1}{k}-1} \Big\{(|q|^p + 2) \left(1+{\frac{1}{ |\beta|}}\right)^p\cdot\\
\int_{B_{3r}}|\nabla \eta|^p\U^p\omega(x)\,dx+\\
+
\left[{\frac{1}{ \phi^{-1}\left(\frac{V}{\omega};|q|^{-p}\left(1+{\frac{1}{  |\beta|}}\right)^{-p}\right)} }\right]^{n+p}
\int_{B_{3r}}\eta^p\U^p\omega(x)\,dx\Big\}
\end{multline}
where $c$ is a positive constant independent of $\omega$.

Now we choose the function $\eta$. Let $r_1$ and $r_2$ be real numbers such that $r\le r_1 <r_2 \le
2r$ and let the function $\eta$ be chosen so that $\eta(x) =1$ in
$B_{r_1}$, $0\le \eta(x) \le 1$ in $B_{r_2}$, $\eta(x)=0$ outside
$B_{r_2}$, $|\nabla \eta| \le {\frac{c}{r_2-r_1}}$ for some fixed constant $c$.
We have
\begin{multline*}
\left(\int_{B_{r_1}}\U^{kp}\omega(x)\,dx\right)^{\frac{1}{k}} 
\le c \omega(B)^{\frac{1}{k}-1}
\frac{1}{(r_2 -r_1)^p}(|q|^p + 2)\,\cdot
\\
\cdot
\left(1+{\frac{1}{ |\beta|}}\right)^p \left[{\frac{1}{  \phi^{-1}\left(\frac{V}{\omega};|q|^{-p}
\left(1+{\frac{1}{|\beta|}}\right)^{-p}\right)}}\right]^{n+p} \int_{B_{r_2}}\U^p\omega(x)\,dx\,.
\end{multline*}
Setting $\gamma = pq = p + \beta - 1$ and recalling that
$\U(x)=w^q(x)$, we get

\begin{multline}\label{eq:iterazione_negativa}
\Phi (k \gamma,r_1) \ge c^{\frac{1}{\gamma}} \omega(B)^{\frac{1}{\gamma}(\frac{1}{k}-1)}
(|q|^p + 2)^{\frac{1}{  \gamma}}\cdot
\\
\cdot \left[{\frac{1}{ \phi^{-1}\left(\frac{V}{\omega};|q|^{-p}\right)}}\right]^{\frac{n+p}{\gamma}}
{\frac{1}{(r_2-r_1)^{\frac{1} { p}}}} \Phi (\gamma,r_2)\,,
\end{multline}
for negative $\gamma$.
This is the inequality we are going to iterate. 
If $\gamma_i = k^i p_0$ and $r_i = r + {\frac{r}{ 2^i}}$, $i=1,2,\dots$ iteration of \eqref{eq:iterazione_negativa} and use of Lemma \ref{lemmaFilippo} yield
\begin{equation*}
\Phi(-\infty,r)\ge c(p,a,\phi_\frac{V}{\omega},\hbox{diam}\,\Omega) \omega(B_r)^{\frac{1}{ p_0}}\Phi(-p_0,2r)\,.
\end{equation*}
Therefore by H\"older inequality, 
\begin{equation*}
\Phi(p_0^{\prime},2r)
\le
\Phi(p_0,2r)\omega(B_r)^{{\frac{1} {p_0^{\prime}}}-{\frac{1}{ p_0}}}\,,
\quad p^\prime_0\le p_0\,.
\end{equation*}
So we obtain
\begin{equation*}
\omega^{-1}(B_{2r})\Phi(1,2r)\le c\Phi(-\infty,r)
\end{equation*}
where $c\equiv c(p,a,\phi_\frac{V}{\omega},\hbox{diam}\,\Omega)$ and the result follows.
\end{proof}

The next result is a weak Harnack inequality for weak subsolutions. The proof is essentially the same of the proof of the previous one.

\begin{teo}
Let $u$ be a non negative weak subsolution of equation \eqref{equazione} in $\Omega$ satisfying \eqref{ipo:strutt} 
and \eqref{ipo:terminiinferiori}. 
Let $B_r$ be a ball such that $B_{3r} \Subset \Omega$ and let $M$ be a constant such that $ u\le M$ in $B_{3r}$. 
$$
\max_{B_r}u\le c\left\{ \omega^{-1} (B_{2r})\int_{B_{2r}}u\,\omega dx+h(r)\right\}\,
$$
where 
$\displaystyle{
h(r)={\left[
\phi\left(\left(\frac{e}{\omega}\right)^{\frac{p}{ p-1}};3r\right)
+
\phi\left(\frac{g}{\omega};3r\right) \right]}^{\frac{1}{p}}+
\left[ \phi\left(\frac{f}{\omega};3r\right) \right]^{\frac{1}{ p-1}}.
}$
\end{teo}

If we take a non negative weak solution, we can put together the two previous results.
\begin{teo}
Let $u$ be a non negative weak solution of equation \eqref{equazione} in $\Omega$ satisfying \eqref{ipo:strutt} 
and \eqref{ipo:terminiinferiori}. 
Let $B_r$ be a ball such that $B_{3r} \Subset \Omega$ and let $M$ be a constant such that $ u\le M$ in $B_{3r}$. 
Then there exists $c$ depending on $n$, $M$, $a_0$, $b_0$, $p$ and the weight $v$ such that 
$$
\max_{B_r}u\le c\left\{ \min_{B_r}u+h(r)\right\}\,
$$
where 
$\displaystyle{
h(r)={\left[
\phi\left(\left(\frac{e}{\omega}\right)^{\frac{p}{ p-1}};3r\right)
+
\phi\left(\frac{g}{\omega};3r\right) \right]}^{\frac{1}{p}}+
\left[ \phi\left(\frac{f}{\omega};3r\right) \right]^{\frac{1}{ p-1}}.
}$
\end{teo}

Now, as a simple consequence of Harnack inequality, we get some regularity results for weak solutions of \eqref{equazione}.
The proof is an immediate consequence of Harnack inequality so we omit it.

\begin{teo}
Let $u$ be a weak solution of equation \eqref{equazione} in $\Omega$ satisfying \eqref{ipo:strutt} 
and \eqref{ipo:terminiinferiori}. 
Let $B_r$ be a ball such that $B_{3r} \Subset \Omega$ and let $M$ be a constant such that $ u\le M$ in $B_{3r}$. 
Then $u$ is continuous in $\Omega$.
\end{teo}

If we assume more restrictive assumptions on the lower order terms we obtain the following refinement of the previous one.

\begin{teo}
Let $u$ be a weak solution of equation \eqref{equazione} in $\Omega$ satisfying \eqref{ipo:strutt}
and
\begin{equation*} 
a, b_0\in \R, \left(\frac{b}{\omega}\right)^{\frac{p}{p-1}}, \left(\frac{b_1}{\omega}\right)^p, \frac{d}{\omega}, \left(\frac{e}{\omega}\right)^{\frac{p}{p-1}}, \frac{f}{\omega}, \frac{g}{\omega}\in L^{1,p-\varepsilon}_v(\Omega)\,,
\quad \varepsilon>0\,.
\end{equation*}
Let $B_r$ be a ball such that $B_{3r} \Subset \Omega$ and let $M$ be a constant such that $ u\le M$ in $B_{3r}$. 
Then $u$ is locally H\"older continuous in $\Omega$.
\end{teo}

\section{Application to non variational degenerate equations}
As an application of the results in the previous section, we prove continuity and  H\"older continuity estimates for the gradient of solutions of some quasilinear non variational elliptic equations. 
The equations we are going to consider are degenerate elliptic with respect to a power of a strong $A_\infty$ weight.

Let $\Omega$ be a bounded domain in $\R^n$ ($n\geq 3$) and  $v$ a strong $A_\infty$ weight in $\R^n$.
We consider the equation
\begin{equation}\label{eqnonvariazionale}
Qu=a^{ij}(x,u,\nabla u)u_{x_ix_j}+b(x,u,\nabla u)=0\,, \quad \hbox{\rm in} \ \Omega\,.
\end{equation}
We assume the functions $a^{ij}(x,u,p)$, $b(x,u,p)$ to be differentiable in $\Omega\times\R\times\R^n$ and
the following degenerate ellipticity condition

\begin{equation} \label{degenerateellipticity}
\exists \lambda>0 \, : \,
\lambda^{-1}\omega(x)\, |\xi|^2
\le
a^{ij}(x,u,p) \, \xi_i \xi_j 
\le
\lambda \omega(x) \, |\xi|^2
\end{equation}
for a.e. $x$ in $\Omega$, $\forall u\in \R$, $p\in \R^{n}$ and $\forall \xi\in\R^{n}$
where $\omega=v^{1-\frac{2}{n}}$.

\begin{defi}
Let $u$ be a $H^{2,2}_{v,{\rm loc}}(\Omega)$ function. We say that $u$ is a $H^{2,2}_{v,{\rm loc}}(\Omega)$ solution if there exists a ball $B$ in $\Omega$ such that $u$ satisfies \eqref{eqnonvariazionale} almost everywhere in $B$.
\end{defi}

Now we prove the following

\begin{teo}
Let $\Omega$ be a bounded domain in $\R^n$ and let $u$ be a $H^{2,2}_{v,{\rm \loc}}(\Omega)$ solution of equation \eqref{eqnonvariazionale} satisfying  \eqref{degenerateellipticity}.
We set 
$$
f(x)=
\sup\{|a_u^{ij}(x,u,\nabla u)|, |a_x^{ij}(x,u,\nabla u)|, |b(x,u,\nabla u)| \}
$$
and assume that $\left(\dfrac{f}{\omega}\right)^2\in S^\prime_v(\Omega)$ and for any ball $B_r$ in $\Omega$ there exist positive constants $M$ and $K$ such that 
$|\nabla u|\le M$ and $\dfrac {|a_p^{ij}(x,u(x),\nabla u(x))|}{\omega(x)} \le K$ in $B_r$.

Then there exists $0<\sigma \leq 1$ depending on 
$\lambda$, $n$, $M$, $K$ and the weight $v$ such that for any $0<\rho<r$
$$
\osc_{B_\rho}u_{x_l}\le c \left[\phi\left(\dfrac{f^2}{\omega^2},3\rho\right)\right]^{\sigma/2}  ,\quad l=1,2,...,n
$$
where $\phi$ is the function in the Definition \ref{def:stummel}.
\end{teo}

\begin{proof}
For  $k=1,2,...,n$ we have   
\begin{equation}\label{defdisoluzione}
\int_{B_r}(a^{ij}(x,u,\nabla u)u_{x_ix_j}+b(x,u,\nabla u))\varphi_{x_k}=0\quad \forall \varphi\in H^{1,2}_{0,v}(B_r) \,.
\end{equation}
There is no loss of generality in assuming $u\in C^3(\Omega)$. This assumption can be removed later via a density argument.
Since
\begin{multline}\label{5}
\int_{B_r}a^{ij}(x,u,\nabla u)u_{x_ix_j}\varphi_{x_k}dx=\int_{B_r}(-a^{ij}_{x_k}u_{x_ix_j}\varphi-a^{ij}_uu_{x_k}u_{x_ix_j}\varphi\\
-a^{im}_{p_j}u_{x_jx_k}u_{x_ix_m}\varphi-a^{ij}u_{x_ix_jx_k}\varphi )dx
\end{multline}
and
\begin{multline}\label{6}
-\int_{B_r}a^{ij}u_{x_ix_jx_k}\varphi dx=\int_{B_r}(a^{ij}u_{x_jx_k}\varphi_{x_i}+a^{ij}_{x_i}u_{x_jx_k}\varphi\\
+a^{ij}_u u_{x_i}u_{x_jx_k}\varphi+a^{ij}_{p_m}u_{x_mx_i}u_{x_jx_k}\varphi) dx
\end{multline}
\noindent
then \eqref{defdisoluzione} reads
\begin{equation}\label{2}
\int_{B_r}\{a^{ij}u_{x_kx_j}\varphi_{x_i}+(a^{ij}_{m}u_{x_mx_i}u_{x_jx_k}+a^ju_{x_kx_j}+b^{ij}_ku_{x_ix_j})\varphi+b\varphi_{x_k}\}dx=0
\end{equation}
where
$$
a^{ij}_m=a^{ij}_{p_m}(x,u,\nabla u)-a^{im}_{p_j}(x,u,\nabla u) 
$$
$$
a^j=a^{ij}_u(x,u,\nabla u)u_{x_i}+a^{ij}_{x_i}(x,u,\nabla u)
$$
$$
b^{ij}_k=-a^{ij}_u(x,u,\nabla u)u_{x_k}-a^{ij}_{x_k}(x,u,\nabla u)\,.
$$

We choose $\varphi=u_{x_k}\eta(x)$ as test function in \eqref{2}, where $\eta\ge 0$, $\eta\in C^{1}_0(B_r)$ so we get
\begin{multline}\label{3}
\int_{B_r}\{a^{ij}u_{x_kx_i}u_{x_kx_j}\eta+\frac{1}{2} a^{ij}v_{x_j}\eta_{x_i}+ \frac{1}{2}a^{ij}_m u_{x_mx_i}v_{x_j}\eta\\
+\frac{1}{2}a^jv_{x_j}\eta+b^{ij}_ku_{x_ix_j}u_{x_k}\eta+b\Delta u\eta+bu_{x_k}\eta_{x_k}\}dx=0
\end{multline}
where we set $v=|\nabla u|^2$.

Let $\gamma>0$ and set $ w=w^+_l=\gamma  u_{x_l}+v$, $l=1,...,n$.
>From \eqref{3} and \eqref{2} we obtain
\begin{multline*}
\int_{B_r}\{a^{ij}u_{x_ix_k} u_{x_kx_j}\eta+(\frac{1}{2} a^{ij}w_{x_j}+bu_{x_i}+\frac{1}{2}\gamma b\delta^l_i)\eta_{x_i}\}dx=\\
-\int_{B_r}\{\frac{1}{2}a^{ij}_m u_{x_mx_i}w_{x_j}+\frac{1}{2}a^jw_{x_j}+(\frac{1}{2}\gamma b^{ij}_l+ b^{ij}_ku_{x_k}+b\delta^j_i)u_{x_ix_j}\}\eta 
dx
\end{multline*}

and then
\begin{multline}\label{disegdivergenza}
\int_{B_r}(a^{ij}w_{x_j}+2bu_{x_i}+\gamma b\delta^l_i)\eta_{x_i}dx
\le
\\
\le
\int_{B_r}\big\{-2\lambda^{-1}\omega(x)|D^2u|^2-a^{ij}_m u_{x_mx_i}w_{x_j}-a^jw_{x_j}-
\\
-
(\gamma b^l_{ij}+2 b^k_{ij}u_{x_k}+2b\delta^j_i)u_{x_ix_j}\big\}\eta dx\le
\\
\le
\int_{B_r}\Big\{-2\lambda^{-1}\omega(x)|D^2u|^2+(\sum_{m,i}(a^{ij}_mw_{x_j})^2)^{\frac{1}{2}}|D^2u|+
\\
+|\nabla w|f(x)+|D^2u|f(x)\Big\}\eta dx \le
\\
\le
\int_{B_r}\Big\{-2\lambda^{-1}\omega(x)|D^2u|^2+\lambda^{-1}\omega(x)|D^2u|^2+\frac{\lambda}{\omega(x)}\sum_{m,i}(a^{ij}_mw_{x_j})^2+
\\
+\lambda^{-1}\omega(x)|\nabla w|^2+\frac{2\lambda f^2}{\omega(x)}+\lambda^{-1}\omega(x)|D^2u|^2\Big\}\eta dx\le
\\
\le c(K,\lambda)\int_{B_r}\Big\{\omega(x)|\nabla w|^2+\frac{ f^2}{\omega(x)}\Big\}\eta dx\,.
\end{multline}

The previous inequality shows that $w(x)=w^+_l(x)$ is a local weak subsolution of the equation 
\begin{equation}\label{eqdivergenza}
- \left(\tilde a_{ij}w_{x_i}\right)_{x_j}-c(K,\lambda)\omega |\nabla w|^2
=
\dfrac{f^2}{\omega}- (F_i(x))_{x_i}
\end{equation}
where
$\displaystyle{
\tilde a_{ij}(x)=a^{ij}(x,u(x),\nabla u(x))
}$
and
$$
F_i(x)=-2b(x,u(x),\nabla u(x))u_{x_i}(x)-\gamma b(x,u(x),\nabla u(x))\delta^l_i\,.
$$

We note that $|F_i\,| \le c(M)f $, $i=1,2,...,n$. This implies $\left(\dfrac{F_i}{\omega}\right)^2\in S^\prime_v(B_r)$ and 
$$
\left\|\left(\frac{F_i}{\omega}\right)^2\right\|_{S^\prime_v(B_r)}\le \left\|\left(\frac{f}{\omega}\right)^2\right\|_{S^\prime_v(B_r)}\,.
$$ 

Now fix $0<\rho<{\rm min}\{1,\frac{1}{3}r\}$ and choose $1\le h\le n$ such that 
$$
\osc_{B_{3\rho}}u_{x_h}\ge \osc_{B_{3\rho}}u_{x_l} \quad  \forall l=1,2,...,n\,.
$$

Now we fix $\gamma$ a sufficiently large positive number. It turns out that a convenient choice is $\gamma=10n M$.
Then we have
\begin{multline*}
\osc_{B_{3\rho}}w^+_h 
\le 
\osc_{B_{3\rho}}(10nM u_{x_h})+\osc_{B_{3\rho}}|\nabla u|^2\le 
\\
\le 10nM \osc_{B_{3\rho}} u_{x_h}+\osc_{B_{3\rho}}(\sum_{i=1}^n u_{x_i}^2)\le \\
\le
10nM \osc_{B_{3\rho}} u_{x_h}+2M  \osc_{B_{3\rho}}(\sum_{i=1}^n u_{x_i})\le 12 nM \osc_{B_{3\rho}} u_{x_h}
\end{multline*}
and 
$$
\osc_{B_{3\rho}}w^+_h
\ge
10nM \osc_{B_{3\rho}} u_{x_h}-\osc_{B_{3\rho}}(\sum_{i=1}^n u_{x_i}^2)
\ge
8nM \osc_{B_{3\rho}} u_{x_h}\,.
$$
Putting together the previous inequalities we obtain
\begin{equation} \label{osc}
8nM \osc_{B_{3\rho}} u_{x_h}
\le
\osc_{B_{3\rho}}w^+_h
\le
12 nM \osc_{B_{3\rho}} u_{x_h}\,.
\end{equation}

The same argument applies to the function $w^-_h = -\gamma u_{x_h} + v$.
Arguing in the same way we get

\begin{equation} 
8nM \osc_{B_{3\rho}} u_{x_h}
\le
\osc_{B_{3\rho}}w^-_h 
\le
12 nM \osc_{B_{3\rho}} u_{x_h}\,.
\end{equation}

The functions $\sup_{B_{3\rho}} w^+_h - w^+_h $ and $\sup_{B_{3\rho}} w^-_h - w^-_h $ are supersolutions of \eqref{eqdivergenza} which is non linear because of the quadratic term in the gradient. 
However, we may apply the results in the previous section taking $p=2$. 
Then, from Theorem \ref{harnackinterno} we get 
\begin{equation}\label{che.}
\omega^{-1} (B_{2\rho})\int_{B_{2\rho}}(\sup_{B_{3\rho}} w^+_h - w^+_h )\omega dx
\le 
c (\sup_{B_{3\rho}} w^+_h - \sup_{B_{\rho}} w^+_h +h(\rho))
\end{equation}
and
\begin{equation}\label{chemeno}
\omega^{-1} (B_{2\rho})\int_{B_{2\rho}}(\sup_{B_{3\rho}} w^-_h - w^-_h)\omega dx
\le 
c (\sup_{B_{3\rho}} w^-_h - \sup_{B_{\rho}} w^-_h +h(\rho))
\end{equation}

where $h(\rho)=\left[\phi\left(\frac{f^2}{\omega^2},3\rho\right)\right]^{1/2}$.

As a conseguence of \eqref{osc} we have 
\begin{multline}\label{osc1}
\sup_{B_{3\rho}}w^+_h - w^+_h +\sup_{B_{3\rho}}w^-_h - w^-_h 
=
\sup_{B_{3\rho}}w^+_h + \sup_{B_{3\rho}}w^-_h -2v
\ge
\\
\ge
 \sup_{B_{3\rho}}w^+_h + \sup_{B_{3\rho}}w^-_h -2\sup_{B_{3\rho}}v
\ge
\\
\ge 
\sup_{B_{3\rho}}(10nMu_{x_h})-\inf_{B_{3\rho}}(10nMu_{x_h})+2\inf_{B_{3\rho}}v
-2\sup_{B_{3\rho}}v
\ge 
\\
\ge 
10nM \osc_{B_{3\rho}} u_{x_h}-4nM\osc_{B_{3\rho}} u_{x_h}
\ge
\frac{1}{2}\osc_{B_{3\rho}}w^{+}_h \quad \forall x\in B_{3\rho}
\end{multline}

and in the same way we see that

\begin{equation} \label{osc2}
\sup_{B_{3\rho}}w^+_h - w^+_h +\sup_{B_{3\rho}}w^-_h - w^-_h
\ge
\frac{1}{2}\osc_{B_{3\rho}}w^{-}_h \quad \forall x\in B_{3\rho}\,.
\end{equation}

By \eqref{osc1} and \eqref{osc2}, the following inequalities
$$
\begin{cases}
\dfrac{1}{4} \osc_{B_{3\rho}} w^+_h 
>
\omega^{-1}(B_{2\rho}) \int_{B_{2\rho}} (\sup_{B_{3\rho}}w^+_h - w^+_h)\omega(x)\,dx
& \\
& \\
\dfrac{1}{4} \osc_{B_{3\rho}} w^-_h
>
\omega^{-1}(B_{2\rho}) \int_{B_{2\rho}} (\sup_{B_{3\rho}}w^-_h - w^-_h)\omega(x)\,dx
& \\
\end{cases}
$$

cannot be both true at the same time.
Let us suppose that 
$$
\dfrac{1}{4} \osc_{B_{3\rho}} w^+_h
\leq
\omega^{-1}(B_{2\rho}) \int_{B_{2\rho}} (\sup_{B_{3\rho}}w^+_h - w^+_h)\omega(x)\,dx\,.
$$

By \eqref{che.} we are able to give an estimate for the oscillation of $w^+_h$, i.e. 
$$
\osc_{B_{3\rho}}\, w^+_h\le c (\sup_{B_{3\rho}} w^+_h - \sup_{{B_{\rho}}} w^+_h + h(\rho))
\le
$$
$$
\leq
c (\osc_{B_{3\rho}}\, w^+_h - \osc_{B_{\rho}}\, w^+_h + h(\rho))
$$
from which
$$
\osc_{B_{\rho}}\, w^+_h
\le
(1-1/c)\osc_{B_{3\rho}}\, w^+_h + h(\rho)\,.
$$
Now we can apply Lemma \ref{lemma-stampacchia}. There exist postive constants $\sigma\le 1$, $\overline \rho$  and $k$ such that 
$$
\osc_{B_{\rho}}\, w^+_h
\le
k h^\sigma(\rho)\quad \forall \rho<\overline{\rho}\,.
$$
Then we get
$$
\osc_{B_{\rho}}\, u_{x_l}\le c k h^\sigma(\rho)\quad \forall \rho<\overline{\rho} \quad\forall l=1,...n
$$
and the proof is complete.
\end{proof}

Refining our assumptions  - as a consequence  of Proposition \ref{embedding_Stummel_Morrey} - we get 

\begin{teo}
Let $\Omega$ be a bounded domain in $\R^n$ and $u$ be a $H^{2,2}_{v,{\loc}}(\Omega)$ solution of equation \eqref{eqnonvariazionale} satisfying  \eqref{degenerateellipticity}.
We set 
$$
f(x)=
\sup\{|a_u^{ij}(x,u,\nabla u)|, |a_x^{ij}(x,u,\nabla u)|, |b(x,u,\nabla u)| \}\,.
$$
Assume that $\left(\dfrac{f}{\omega}\right)^2\in L^{1,p-\varepsilon}_v (\Omega)$,  for $0<\varepsilon <p$, and for any ball $B_r$ in $\Omega$ there exist positive constants $M$ and $K$ such that 
$|\nabla u|\le M$ and $\dfrac {|a_p^{ij}(x,u(x),\nabla u(x))|}{\omega(x)} \le K$ in $B_r$.

Then there exists $0<\alpha < 1$ depending on $\lambda$, $n$, $M$, $K$ and the weight $v$ such that for any $0<\rho<r$
$$
\osc_{B_\rho}u_{x_l}
\le
c \rho^\alpha  ,\quad l=1,2,...,n\,.
$$
\end{teo}


\begin{thebibliography}{99}

\bibitem{BUCKLEY} 
S. M. Buckley, {\em Inequalities of {J}ohn--{N}irenberg type in doubling spaces}, J. Anal. Math. {\bf 79}, 215--240 (1999).

\bibitem{DS}
G. David - S. Semmes, {\em Strong $A_\infty$ weights, Sobolev inequalities and quasiconformal mappings, Analysis and Partial Differential Equations}, Lecture notes in Pure and Applied Mathematics {\bf 122}, Marcel Dekker, (1990).

\bibitem{DIFAZIO-PADOVA}
G. Di Fazio, {\em  {H}\"older continuity of solutions for some {S}chr\"odinger equations}, Rend. Sem. Mat. Univ. Padova, {\bf 79}, 173-183,(1988).

\bibitem{dfz} 
G. Di Fazio - M.S.Fanciullo - P. Zamboni, {\em Harnack inequality and smoothness for quasilinear degenerate elliptic equations}, Journal of Differential Equations, {\bf 245}, n. 10, 2939-2957 (2008).

\bibitem{dz} 
G. Di Fazio - P. Zamboni, {\em Regularity for quasilinear degenerate elliptic equations}, Math. Z., {\bf 253}, 787-803 (2006).

\bibitem{FKS} 
E. Fabes - C. Kenig - R. Serapioni, {\em The local regularity of solutions of degenerate elliptic equations}, Comm. PDE, {\bf 7(1)}, 77-116 (1982).     

\bibitem{FGW}
B. Franchi - C. Gutierrez - R. Wheeden, {\em Weighted Sobolev-Poincar\'e inequalities for Grushin type operators}, Comm. PDE, {\bf 19}, 523--604 (1994).

\bibitem{FRANCHI-LINCEI}
B.Franchi, C.Gutierrez, R.Wheeden,
{\em Two-weight Sobolev-Poincar\'e inequalities and Harnack inequality for a class of degenerate elliptic
operators},
Atti Accad. Naz. Lincei Cl. Sci. Fis. Mat. Natur. Rend. Lincei, 
(9) Mat. Appl. 5, no. 2, 167--175 (1994).


\bibitem{FRANCHI-HAILASZ}
B. Franchi - P. Haj\l asz, {\em  {H}ow to get rid of one of the weights in a two-weight Poincar\'e inequality?}, Annales Polonici Mathematici, {\bf LXXIV}, 97-103 (2000).

\bibitem{FSSC}
B. Franchi - R. Serapioni - F. Serra Cassano, { \em  Approximation and imbedding theorems for weighted Sobolev spaces associated with Lipschitz continuous vector fields},  BUMI (7) {\bf 11}-B, 83-117 (1997). 

\bibitem{gt} 
D. Gilbarg - N. S. Trudinger, {\em Elliptic Partial Differential Equations of Second Order}, Springer-Verlag, Berlin (1983).   

\bibitem{GUT} 
C. E. Gutierrez, {\em {H}arnack's inequality for degenerate {S}chr\"odinger operators}, TAMS {\bf 312}, 403-419 (1989).

\bibitem{HEI-KOSK}
J. Heinonen - P. Koskela, {\em  {W}eighted {S}obolev and {P}oincar\'e inequalities and quasiregular mappings of polynomial type},  Math.Scand. {\bf 77}, 251-271(1995).

\bibitem{Lieberman} 
G.M. Lieberman, {\em {S}harp form of estimates for subsolutions and supersolutions of quasilinear elliptic equations involving measures}, Comm. PDE {\bf 18}, 1191-1212 (1993).

\bibitem{RAG-ZAM}
M. A. Ragusa - P. Zamboni, {\em Local regularity of solutions to quasilinear elliptic equations with general structure}, Commun. Appl. Anal.,{\bf 3}, no. 1, 131-147 (1999).

\bibitem{SEMMES}
S. Semmes, {\em  {M}etric spaces and mappings seen at many scales},  M.Gromov - Metric structures for Riemannian and non -- Riemannian spaces,
Progress in Mathematics {\bf 152}, Birkha\"user, Boston, MA (1999).

\bibitem{SC}
F. Serra Cassano, {\em  On the local boundedness of certain solutions for a class of degenerate elliptic equations}, BUMI (7), {\bf 10}-B, 651-680(1996).


\bibitem{t} 
N. S. Trudinger, {\em On Harnack type inequalities and their application to quasilinear elliptic equations}, CPAM {\bf XX}, 721-747 (1967).

\bibitem{vz} C. Vitanza - P. Zamboni, {\em Necessary and sufficient  conditions for h\"{o}lder continuity of solutions of degenerate schr\"{o}dinger operators}, Le Matematiche {\bf 52}, n.2, 393-409 (1997). 

\bibitem{z} P. Zamboni, {\em H\"{o}lder continuity for solutions of linear degenerate elliptic equations under minimal assumptions}, Journal of Differential Equations, {\bf 182}, n.1, 121-140 (2002). 



\end{thebibliography}
\end{document}